\documentclass[11pt]{article}
\usepackage[T1]{fontenc}
\usepackage{lmodern}
\usepackage[a4paper,margin=27mm]{geometry}
\usepackage{amsmath,amssymb,amsthm,mathtools}
\usepackage{microtype}
\usepackage[colorlinks=true,linkcolor=blue,citecolor=blue,urlcolor=blue]{hyperref}
\hypersetup{
 pdftitle={Irrationality exponents of logarithms of positive rational numbers},
 pdfauthor={Jingwen Liu; Kai Jiang; Pingwen Zhang},
 pdfsubject={Diophantine approximation and moving-centre interpolation},
 pdfkeywords={irrationality exponent, logarithms, interpolation determinants, weighted interpolation}
}
\newtheorem{theorem}{Theorem}
\newtheorem{proposition}[theorem]{Proposition}
\newtheorem{corollary}[theorem]{Corollary}
\newtheorem{lemma}[theorem]{Lemma}
\theoremstyle{remark}
\newcommand{\C}{\mathbb C}
\newcommand{\Z}{\mathbb Z}
\newcommand{\N}{\mathbb Z_{\geq0}}
\newcommand{\wa}[1]{\langle w,#1\rangle}
\newcommand{\ord}{\operatorname{ord}}

\newcommand{\lcm}{\operatorname{lcm}}

\title{Irrationality exponents of logarithms\protect\\
of positive rational numbers}
\author{Jingwen Liu\thanks{\raggedright School of Mathematics and Statistics,
Wuhan University, Wuhan, Hubei, 430072, China.
Email: \href{mailto:jingwen_liu@whu.edu.cn}{\texttt{jingwen\_liu@whu.edu.cn}}.}
\and Kai Jiang\thanks{\raggedright Wuhan Institute for Math \& AI,
Wuhan University, Wuhan, Hubei, 430072, China.
Email: \href{mailto:kaijiang@whu.edu.cn}{\texttt{kaijiang@whu.edu.cn}}.}
\and Pingwen Zhang\thanks{\raggedright School of Mathematics and Statistics,
Wuhan University, Wuhan, 430072, China; and School of Mathematical Sciences,
Peking University, Beijing, 100871, China.
Email: \href{mailto:pzhang@pku.edu.cn}{\texttt{pzhang@pku.edu.cn}}.}}
\date{7 October 2026}
\begin{document}
\begingroup
\makeatletter
\let\@fnsymbol\@arabic
\makeatother
\maketitle
\endgroup
\begin{abstract}
We study rational approximation to logarithms of positive rational
numbers and give a separated-weight determinant argument for
$\mu(\log r)=2$ whenever $r\in\mathbb Q_{>0}\setminus\{1\}$.
The geometric input is an interpolation theorem for logarithmic jets
at points with distinct multiplicative coordinates. Its successive weight
thresholds are independent of the interpolation centres; the eventual
degree threshold may depend on them. A local intersection-length bound
and a curve inequality lead to surjectivity through a blow-up and Serre
vanishing. For a fixed rational argument, this produces a nonzero rational
determinant. Clearing denominators gives an arithmetic lower bound,
while an exact row translation and repeated Taylor indices yield a
contradictory analytic upper bound. The consequences include nonzero
rational linear combinations of such logarithms, rational affine changes,
and logarithms of positive numbers with a rational positive integer power.
We also give a bounded-fibre interpolation extension, an explicit bound
for the final analytic degree, and a counterexample to interpolation
under the volume condition alone.
\end{abstract}

\section{Introduction}\label{sec:introduction}
\hspace*{2em}The theory of irrationality measures has a classical foundation in the
rational approximation of algebraic numbers. Liouville's 1844 estimate bounded the quality of
approximation in terms of the algebraic degree. Subsequent refinements
culminated in Roth's theorem in 1955, which established irrationality
exponent $2$ for every irrational algebraic number \cite{roth}.
This result provides a benchmark for the study of other explicitly given
numbers. The logarithm values considered here require their own
approximation argument; Roth's theorem supplies historical context rather
than an input to our proof.

By the 1980s, quantitative work on logarithm values had developed explicit
constructions of small linear forms with controlled integer coefficients.
In 1987, Rukhadze obtained an improved irrationality measure for $\log 2$
\cite{rukhadze87}, while Rhin developed polynomial constructions related
to Pad\'e approximation to obtain effective irrationality measures
\cite{rhin87}. These methods compare the growth of the coefficients with
the decay of the linear forms, with the arithmetic of their denominators
playing a central role.

Interpolation determinants provided a complementary development.
In 1991, Laurent used the comparison of arithmetic lower bounds and
analytic upper bounds for determinants in a proof of the six exponentials
theorem \cite[Section 6]{laurent91}. His 2000 work on exponential
polynomials developed Taylor expansions and upper estimates for
interpolation determinants \cite{laurent00}. These results made the
determinant itself a useful object for combining analytic estimates with
arithmetic nonvanishing.

Zudilin's 2004 survey \cite{zudilin} brought together integral,
hypergeometric, and simultaneous-approximation methods for logarithm
values, explaining the estimates for $\log 2$, $\pi$, and $\log 3$
available at that time. The present paper follows the determinant approach
and supplies a geometric interpolation statement adapted to centres that
depend on the chosen rational approximations. The required independence
of the weight thresholds from the centres is proved in
Section~\ref{sec:geometry}.

We study rational approximations to the real logarithms of positive
rational numbers. The argument combines moving-centre polynomial
interpolation with arithmetic and analytic estimates for a single
determinant. All logarithms are natural. For an irrational real number $x$, its
irrationality exponent is
\[
 \mu(x)=\sup\left\{\nu>0:
  \left|x-\frac pq\right|<q^{-\nu}
  \text{ for infinitely many }p\in\Z,\ q\in\Z_{>0}\right\}.
\]
Dirichlet's theorem gives $\mu(x)\geq2$ for every irrational real $x$.
The main result asserts that this lower bound is attained throughout
the family considered here.
\begin{theorem}[Logarithms of positive rational numbers]\label{thm:main}
For every $r\in\mathbb Q_{>0}\setminus\{1\}$ and every real $\nu>2$,
there is a positive integer $Q(r,\nu)$ such that
\[
 \left|\log r-\frac pq\right|>q^{-\nu}
 \qquad(p\in\Z,\ q\in\Z,\ q\geq Q(r,\nu)).
\]
Consequently $\log r$ is irrational and $\mu(\log r)=2$.
\end{theorem}
The fractions in this statement need not be reduced. The threshold
depends on both $r$ and $\nu$, and no effective value for it is obtained
here. In particular, the assertion does not supply a uniform positive
lower bound of the form $c(r)/q^2$.
Sections~\ref{sec:parameters}--\ref{sec:contradiction}
prove this theorem directly for a fixed arbitrary $r$, using the
interpolation theorem proved in Section~\ref{sec:geometry}.
The connection with the argument for $\pi$ in \cite{source} is provided
by the exponential map. For $r>0$, its fibre over $r$ is
\[
 \{z\in\C:e^z=r\}=\log r+2\pi i\Z.
\]
Here $\log r$ is a real preimage of $r$, and $2\pi i$ is the period
separating its preimages. The identity
$e^{j\log r+2\pi i k}=r^j$ for $j,k\in\Z$ shows that $j\log r$
maps to the distinct rational values $r^j$, whereas $2\pi i k$ maps to $1$.

Both constructions use the curve $z\mapsto(e^z,z,\ldots,z)$.
Writing $y_0=e^{z_0}$ and $z=z_0+\log(1+t)$ gives the common local model
\[
 Y=y_0(1+t),\qquad
 X_i=z_0+\log(1+t)+u_i\quad(1\leq i\leq m).
\]
The $u_i$ are transverse displacements; along the curve,
$dX_i=dY/Y$. Replacing $z_0=j\log r$ by $jp_i/q_i$, or
$z_0=2\pi i j$ by $2ij p_i/q_i$ in the additive coordinates, gives
the respective interpolation centres. The shared local model explains
the parallel Taylor expansions and analytic determinant estimates.

The geometric input depends on the centre configuration. The present
proof adapts the separated-weight method of \cite[Section 2]{source}.
Since the coordinates $Y=r^j$ are distinct, a curve with constant $Y$
meets at most one centre, allowing arbitrary additive coordinates in
Theorem~\ref{thm:MI}. For $\pi$, all centres lie in $Y=1$; interpolation
requires additive-coordinate separation and $K\theta^m<1$, as in the
case $\kappa=K$ of Appendix~\ref{app:fibres}. These hypotheses must be
checked separately despite the common analytic model. For $r=a/b$,
clearing powers of $b$ introduces an additional arithmetic cost,
accounted for in the parameter choices below.

\paragraph{Proof outline.}
The argument constructs a nonzero rational determinant and shows that
its arithmetic lower bound is incompatible with its analytic upper bound.
The steps, and their order of dependence, are as follows.
\begin{enumerate}
\item \textit{Assume excessively good approximations and fix the parameters.}
Fix $r=a/b$ and a rational $\nu>2$, and suppose that
$|\log r-p/q|\leq q^{-\nu}$ holds for arbitrarily large denominators.
Section~\ref{sec:parameters} first chooses the dimension $m$, the number
$K$ of centres, and the auxiliary parameters. It then selects $m$
approximations $p_i/q_i$ successively, with separated weights
$w_i=\lceil\log q_i\rceil$. The weight thresholds in the moving-centre
interpolation theorem are independent of the centres, so this choice
does not depend circularly on the selected approximations.

\item \textit{Use interpolation to obtain a nonzero determinant.}
Use the centres with coordinates $Y=r^j$ and $X_i=jp_i/q_i$, for
$0\leq j<K$, and form the matrix of prescribed local coefficients of
weighted monomials.
Theorem~\ref{thm:MI}, proved in Section~\ref{sec:geometry}, makes this
matrix surjective for every sufficiently large admissible degree $H$.
Section~\ref{sec:arithmetic} therefore selects a nonzero square minor
$\Delta_H$. This is the role of MI in the determinant argument.

\item \textit{Bound the determinant from below by arithmetic.}
All matrix entries are rational. The row and column scalings in
Section~\ref{sec:arithmetic}, including the factor clearing powers of
$b$, turn $\Delta_H$ into a nonzero integer of absolute value at least
$1$. Tracking these factors gives a quantitative lower bound for
$|\Delta_H|$.

\item \textit{Bound the same determinant from above by analysis.}
The identity $e^{j\log r}=r^j$ relates the centres to the analytic curve
$(e^z,z,\ldots,z)$. Section~\ref{sec:analytic} expands the rows around
this curve. In each determinant term, either most transverse indices are
large, giving a large total approximation-error weight, or many indices
belong to a small finite set. In the latter case, repetition forces high
Taylor orders in every nonzero Taylor contribution, and Cauchy's estimate
gives the required decay.

\item \textit{Compare the bounds and conclude.}
Section~\ref{sec:contradiction} shows that the parameter choices leave a
strict positive gap for every rational $\nu>2$. With all approximations,
weights, and centres fixed, let $H$ tend to infinity through the
admissible degrees. The remaining errors tend to zero, and the upper
bound becomes smaller than the lower bound. This contradiction excludes
arbitrarily large exceptional denominators. Choosing a rational exponent
between $2$ and any given real exponent extends the conclusion to every
real $\nu>2$. Exact rational representations would contradict the same
bound, so $\log r$ is irrational; Dirichlet's theorem then supplies the
opposite inequality $\mu(\log r)\geq2$.
\end{enumerate}

\paragraph{Organization.}
Part~\ref{part:interpolation} introduces the weighted polynomial spaces
and proves the moving-centre interpolation theorem.
Part~\ref{part:main-proof} proves Theorem~\ref{thm:main}: parameters are
chosen first, followed by the arithmetic lower bound, the analytic upper
bound, and their final comparison.
Part~\ref{part:consequences} derives the consequences for rational
linear combinations, affine changes, and rational powers, and discusses
the scope and effectivity of the argument.
Appendices~\ref{app:degree}--\ref{app:volume-control} collect the
finite-degree estimate, the bounded-fibre interpolation extension, and
a counterexample explaining the need for weight separation.

\clearpage
\part{The interpolation theorem}\label{part:interpolation}
\section{Moving-centre interpolation}\label{sec:geometry}
\paragraph{Notation.}
For positive rational weights $W=(w_0,w_1,\ldots,w_m)$, set
\[
 \wa\alpha=\sum_{i=1}^m w_i\alpha_i,\qquad
 V=(v_0,v_1,\ldots,v_m)=(v_0,w_1/\theta,\ldots,w_m/\theta).
\]
Multi-indices have nonnegative integer entries, and their comparisons
are coordinatewise. Write $|\alpha|=\sum_i\alpha_i$ and
$\binom\alpha\beta=\prod_i\binom{\alpha_i}{\beta_i}$.
Let $\mathcal P_W(H)$ be the span of the monomials $Y^hX^\alpha$ with
$w_0h+\wa\alpha\leq H$. Let $\mathcal J_V(H)$ be the quotient of
$\C[[t,u_1,\ldots,u_m]]$ by the monomial ideal of $V$-weight at least $H$.
Thus its retained monomials are precisely $t^su^\beta$ with
$v_0s+\wa\beta/\theta<H$.
For a nonzero formal series $f=\sum_{s,\beta}a_{s,\beta}t^su^\beta$,
its $V$-order is
\[
 \operatorname{ord}_V(f)
 =\min_{a_{s,\beta}\ne0}\left(v_0s+\sum_{i=1}^mv_i\beta_i\right),
 \qquad \operatorname{ord}_V(0)=+\infty.
\]
All asymptotic statements in a degree variable keep the dimension and
weights fixed; all centres are fixed as well before the final degree
limit is taken.

\begin{theorem}[MI: moving-centre interpolation]\label{thm:MI}
Fix positive integers $m,K$, positive rational numbers $w_0,v_0$, and
a rational number $0<\theta<1$, such that
\begin{equation}\label{eq:volume}
 K(w_0/v_0)\theta^m<1.
\end{equation}
There are successive lower thresholds for the positive rational weights
$w_1,\ldots,w_m$, depending only on these fixed data and earlier weights,
with the following property. For any pairwise distinct
$y_0,\ldots,y_{K-1}\in\C^*$ and any $c_{ji}\in\C$, the map
\begin{equation}\label{eq:MI}
 \mathcal P_W(H)\longrightarrow\bigoplus_{j=0}^{K-1}\mathcal J_V(H),
 \qquad P\longmapsto
 \left(P\bigl(y_j(1+t),(c_{ji}+u_i+\log(1+t))_{i=1}^m\bigr)\right)_j
\end{equation}
is onto for every $H=nR$ with $n\geq n_0$. Here the positive integer
$R$ depends only on $W,V$; the threshold $n_0$ may depend on the centres.
The weight thresholds are independent of all centres.
\end{theorem}
We prove this theorem in the next three subsections. No distinctness
condition is imposed on the $c_{ji}$.
\subsection{The local multiplicity input}
We begin with the local multiplicity comparison underlying the
separated-weight argument. We use separate names for degree weights and derivative
costs to prevent an implicit identification of the two.

\begin{lemma}[Uniform local length comparison]\label{lem:length}
In affine $d$-space, assign positive rational degree weights
$\rho_1,\ldots,\rho_d$ to the coordinates. On an open set let
$D_1,\ldots,D_d$ be commuting holomorphic vector fields which form a frame,
with positive costs $\kappa_1,\ldots,\kappa_d$.
Fix $\epsilon>0$.
Suppose $f_1,\ldots,f_r$ are polynomials of weighted degree at most $N>0$,
and $Z$ is a codimension-$k$ irreducible component of their common zeros.
If
\[
 D^\gamma f_l\big|_Z=0
 \quad\text{whenever}\quad\sum_b\kappa_b\gamma_b<\epsilon N,
\]
then, for every coordinate normal basis $A$ and every frame normal basis $B$
at a common general smooth point of $Z$,
\begin{equation}\label{eq:local-sharp}
 \prod_{a\in A}\rho_a\leq k!\epsilon^{-k}\prod_{b\in B}\kappa_b.
\end{equation}
A normal basis means that the indicated vectors map to a basis modulo
$T_xZ$. The constant is independent of the polynomials, $Z$, the point,
and the sizes of the local coordinate neighbourhoods.
\end{lemma}

\begin{proof}
Choose $x$ smooth on $Z$, away from every other component, where the two
normal-basis conditions hold. Flow from $Z$ along the fields in $B$.
Their commutativity gives local coordinates $(z,b)$ in which these fields
are $\partial/\partial b_j$ and $Z$ is $b=0$. The Taylor expansion is
$f_l(z,b)=\sum_\alpha f_{l,\alpha}(z)b^\alpha$, with
$f_{l,\alpha}\equiv0$ for $\sum_{b\in B}\kappa_b\alpha_b<\epsilon N$.

Take the affine slice $S$ obtained by fixing coordinates outside $A$.
It is transverse to $Z$. The derivative of the map $b$ has kernel $T_xZ$,
so its restriction to $T_xS$ is an isomorphism. Thus $b|_S$ are analytic
local coordinates and the other flow coordinates have the form $z=z(b)$.
Every surviving summand
$f_{l,\alpha}(z(b))b^\alpha$ still has weight at least $\epsilon N$:
the coefficient is a power series with nonnegative exponents. No assumption
on coefficient magnitudes or convergence radii is used for this conclusion.

Let $\mathcal O$ be the analytic local ring of $S$ at $x$, let
$J=(f_1|_S,\ldots,f_r|_S)$, and let $\mathfrak a$ be the monomial ideal of
$b$-weight at least $\epsilon N$. We have $J\subseteq\mathfrak a$, so
\[
 \dim_\C(\mathcal O/J)\geq\dim_\C(\mathcal O/\mathfrak a)
 =\#\{\alpha\in\N^k:\textstyle\sum_{b\in B}\kappa_b\alpha_b<\epsilon N\}
 \geq\frac{(\epsilon N)^k}{k!\prod_{b\in B}\kappa_b}.
\]
For the last inequality, the simplex with the displayed strict inequality
is covered by the disjoint unit cubes indexed by the coordinatewise floors
of its points. Every such index is admissible. The number of cubes is at
least the volume of the simplex.

The zero of $J$ is isolated on $S$, hence $J$ is maximal-primary.
Choose $k$ constant linear combinations $g_1,\ldots,g_k$ of the $f_l|_S$
which still have an isolated common zero at $x$. To justify the choice,
at each step avoid the finitely many positive-dimensional components of
the previously chosen equations. None is contained in the common zeros
of all generators of $J$. The forbidden coefficient vectors form finitely
many proper linear subspaces over the infinite field $\C$.
With $J'=(g_1,\ldots,g_k)\subseteq J$, it follows that
$\dim_\C(\mathcal O/J)\leq\dim_\C(\mathcal O/J')$.

The restricted polynomials $g_i$ have degree at most $N$ in the variables
indexed by $A$ with weights $\rho_a$; fixing the other coordinates cannot
increase this degree. Choose an integer $d_0$ for which every $d_0\rho_a$
is a positive integer, and substitute $x_a=\lambda_a+z_a^{d_0\rho_a}$.
Choose $\lambda_a\ne x_a(x)$, so all preimages of $x$ are unramified.
There are $\prod_{a\in A}d_0\rho_a$ of them, and their local intersection
lengths equal that at $x$. Each substituted polynomial has ordinary degree
at most $d_0N$. The affine Bezout bound for isolated zeros, counted with
local intersection length, gives
\[
 \dim_\C(\mathcal O/J')\leq
 \frac{(d_0N)^k}{\prod_{a\in A}d_0\rho_a}
 =\frac{N^k}{\prod_{a\in A}\rho_a}.
\]
Other components elsewhere do not invalidate this isolated-zero bound.
Polynomial, analytic, and completed local rings give the same length here:
all finite-colength quotients are already determined modulo a power of the
maximal ideal. Combining the bounds and cancelling $N^k$ proves the lemma.
\end{proof}

For the intersection-length convention and the isolated-zero Bezout bound
used above, see Mondal~\cite[Section IV.4.2 and Corollary VIII.3]{mondal}.
We will use the following weaker consequence
\begin{equation}\label{eq:local}
 \prod_{a\in A}w_a\leq2k!\epsilon^{-k}\prod_{b\in B}v_b.
\end{equation}
In our application the flows are explicit: $D_0$ acts by
$(Y,X)\mapsto(e^sY,X+s\mathbf1)$, and $D_i$ translates $X_i$.
There is therefore no need for a uniform quantitative flow theorem.
The component $Z$ and the chosen point may change with $N$; every quantity
left in \eqref{eq:local-sharp} is still one of the stated numerical inputs.

\begin{lemma}[The local charts share one global frame]\label{lem:moving-chart}
For every $a_j=(y_j,c_{j1},\ldots,c_{jm})$ with $y_j\ne0$, the formal
coordinates
\[
 t_j=Y/y_j-1,\qquad
 u_{ji}=X_i-c_{ji}-\log(Y/y_j)
\]
give
\[
 D_0=(1+t_j)\partial_{t_j},\qquad
 D_i=\partial_{u_{ji}},\qquad
 D_0=Y\partial_Y+\sum_i\partial_{X_i},\quad D_i=\partial_{X_i}.
\]
The fields on the right are the same at every centre. They preserve
$\mathcal P_W(N)$, and differentiation by $D_a$ decreases local $V$-order
by at most $v_a$. These statements require no bound on the sizes or mutual
distances of the centres.
\end{lemma}
\begin{proof}
The Jacobian of the displayed coordinate change at the centre has
determinant $1/y_j\ne0$. Direct differentiation gives
$D_0t_j=1+t_j$, $D_0u_{ji}=0$, $D_it_j=0$, and
$D_iu_{jl}=\delta_{il}$. The claimed order losses follow on each monomial
$t_j^s u_j^\beta$. The field $Y\partial_Y$ preserves monomial degree,
whereas every $\partial_{X_i}$ decreases it.
Alternatively, the affine change
$(\widetilde Y,\widetilde X)\mapsto(y_j\widetilde Y,c_j+\widetilde X)$
and its inverse preserve each weighted-degree polynomial space.
All logarithms here are formal expansions with constant term zero at
$Y/y_j=1$; no choice of a global branch of $\log Y$ is involved.
There is a different local chart for each centre. This lemma does not
assert the existence of one global change sending all $y_j$ to $1$.
\end{proof}

\subsection{Curves through moving centres}\label{subsec:moving-curves}
Choose rational $\sigma>0$ with
$(1+3\sigma)^{m+1}K(w_0/v_0)\theta^m<1$ and $(1+\sigma)\theta<1$.
Put $\ell_0=m+2$ and
$C_* =\max_{1\leq k\leq m}2k!(\ell_0/\sigma)^k$.
Choose the weights successively so that
\begin{equation}\label{eq:separation}
 \prod_{a\in A}w_a>C_*\prod_{b\in B}v_b
\end{equation}
whenever $A,B\subset\{0,\ldots,m\}$ have the same cardinality between
$1$ and $m$, and the largest positive index where they differ belongs to $A$.
At this largest index, the ratio in \eqref{eq:separation} is proportional
to its $w_i$; larger common indices contribute only the fixed ratio $\theta$.
Thus these are successive lower thresholds independent of all centres.
More explicitly, if $h$ is that largest differing positive index, then
\[
 \frac{\prod_{a\in A}w_a}{\prod_{b\in B}v_b}
 =w_h\theta^{|B\cap\{1,\ldots,m\}|}
 \frac{w_0^{\mathbf1_{0\in A}}}{v_0^{\mathbf1_{0\in B}}}
 \frac{\prod_{a\in A\setminus B,\ 0<a<h}w_a}
      {\prod_{b\in B\setminus A,\ 0<b<h}w_b}.
\]
For each $h$ the coefficient of $w_h$ is known before $w_h$ is chosen.
Taking the maximum of the finitely many resulting thresholds proves the
required order of quantifiers, rather than assuming it from ``large weights''.

For an irreducible algebraic curve $\Gamma$ meeting the affine chart define
on its smooth projective model
\[
 d_W(\Gamma)=\sum_P\max\left(0,-\frac{\ord_PY}{w_0},
              -\frac{\ord_PX_1}{w_1},\ldots,-\frac{\ord_PX_m}{w_m}\right).
\]
At a branch through $(y_j,c_{j1},\ldots,c_{jm})$, set
\[
 h_P=\min\left(\frac{\ord_P(Y/y_j-1)}{v_0},
  \left(\frac{\ord_P(X_i-c_{ji}-\log(Y/y_j))}{v_i}\right)_{i=1}^m\right).
\]
The logarithm is formal at this centre. We claim the curve bound
\begin{equation}\label{eq:curve}
 d_W(\Gamma)\geq(1+\sigma)\sum_{P\text{ through a centre}}h_P.
\end{equation}
An identically zero coordinate has order $+\infty$ and contributes no pole.
The contact $h_P$ is finite: if every entry in its defining minimum were
infinite, $Y$ and every $X_i$ would be constant on the curve.
Curves meeting no centre satisfy the bound immediately.
Suppose the bound fails. To obtain the auxiliary polynomial, note that
weighted simplex counting gives
\[
 \dim\mathcal P_W(N)
 =\frac{N^{m+1}}{(m+1)!\,w_0\prod_iw_i}+O(N^m),\qquad
 \dim\mathcal J_V(T)
 =\frac{T^{m+1}}{(m+1)!\,v_0\prod_iv_i}+O(T^m).
\]
These estimates follow by covering the positive weighted simplex with
unit cubes; the boundary cubes lie in a layer of bounded width and
volume $O(N^m)$ or $O(T^m)$, respectively.
Requiring $V$-order at least $(1+3\sigma)N$ at every centre imposes at most
$K\dim\mathcal J_V((1+3\sigma)N)$ homogeneous linear conditions.
The ratio of this number to $\dim\mathcal P_W(N)$ tends to
$(1+3\sigma)^{m+1}K(w_0/v_0)\theta^m<1$.
Thus, for all sufficiently large integers $N$, a nonzero polynomial
$F_N$ of $W$-degree at most $N$ satisfies all these conditions.
Independence of the conditions is not required for this kernel argument.
On $Y\ne0$ use the commuting frame
\[
 D_0=Y\partial_Y+\sum_i\partial_{X_i},\qquad D_i=\partial_{X_i}.
\]
In the displayed local coordinates $D_0=(1+t)\partial_t$ and
$D_i=\partial_{u_i}$, independently of $y_j$.
The frame preserves the space of weighted-degree-$N$ polynomials.
Every derivative of cost at most $\sigma N$ has branch vanishing at least
$(1+2\sigma)Nh_P$. Its pole degree is at most $Nd_W(\Gamma)$.
The zero--pole equality on the smooth projective model of $\Gamma$
forces each such derivative to vanish on $\Gamma$; see
\cite[Tag \href{https://stacks.math.columbia.edu/tag/02RU}{02RU}]{stacks}.

On $Y\ne0$, use common-zero sets of derivatives of costs at most
$r\sigma N/\ell_0$, for $0\leq r\leq\ell_0$.
Choose a component at the last level containing $\Gamma\cap\{Y\ne0\}$,
then extend it backwards to a nested sequence of components at the earlier
levels. This direction of choice ensures that nesting is possible.
Their dimensions lie between $1$ and $m$, so two consecutive components
coincide. Their affine Zariski closures are components of the corresponding
polynomial zero sets, so Lemma~\ref{lem:length} applies at general points
with $Y\ne0$. On the coincident component $Z$, \eqref{eq:local} with
$\epsilon=\sigma/\ell_0$ applies to the equations of the earlier level.
Indeed, an earlier equation is $D^\alpha F_N$ with cost at most
$r\sigma N/\ell_0$. Applying a further derivative of cost strictly less
than $\sigma N/\ell_0$ gives, by commutativity, an equation vanishing at the
next level. This is exactly the hypothesis of Lemma~\ref{lem:length}.
The same lemma bounds all pairs of normal bases. There are only finitely
many index sets, so the required generic rank conditions can be imposed
on one common nonempty open subset of each particular $Z$.
If a positive coordinate direction is not tangent to $Z$, choose the largest
such index $i$. A coordinate normal basis contains $i$.
Every larger positive coordinate direction is tangent, and the corresponding
frame vector is the same coordinate vector. Such indices belong to neither
kind of normal basis. If a frame normal basis omitted $i$, its comparison
with the chosen coordinate normal basis would therefore have largest
differing positive index $i$, contradicting \eqref{eq:separation}.
Thus every frame normal basis contains $i$.
The other frame directions span the ambient hyperplane
$Q_i=\ker(dX_i-dY/Y)$ and fail to span the normal quotient.
For a hyperplane $Q_i$ in an ambient vector space and a tangent subspace
$T$, its image modulo $T$ fails to be surjective exactly when $T\subseteq Q_i$:
if $T\not\subseteq Q_i$, then $Q_i+T$ is the entire ambient space.
Consequently $T_xZ\subseteq Q_i$ at general smooth points, which proves
$dX_i=dY/Y$ on $Z$.
If $Y$ were nonconstant, restrict to an algebraic curve in $Z$ on which it
remains nonconstant and pass to a smooth complete model. The logarithmic
differential $dY/Y$ has a nonzero residue at a zero or pole of $Y$, whereas
the exact differential $dX_i$ has zero residue everywhere. This is impossible.
For precision, choose a smooth point with $dY\ne0$ and cut by linear
hyperplanes whose tangent intersection is a line on which $dY$ is nonzero.
The curve component through the point has nonconstant $Y$. On its complete
normalization, write $Y=t^r u(t)$ at a zero or pole, with $u(0)\ne0$.
Then $dY/Y=r\,dt/t+du/u$ has residue $r\ne0$. In contrast, differentiating
any Laurent series for $X_i$ produces no $t^{-1}dt$ term.
If every positive coordinate direction is tangent instead, the fact that $Z$ is a proper subvariety
forces its tangent space to be their span, again giving $dY=0$.
Thus $Y$ is constant on $Z$ and on $\Gamma$.

Here the moving-centre argument simplifies: distinct $y_j$ imply that
$\Gamma$ meets at most one centre, say $j$. Some $X_l$ is nonconstant on
$\Gamma$. The function $X_l-c_{jl}$ has pole degree at most $w_l d_W(\Gamma)$ and its zeros over
that centre have total order at least $(w_l/\theta)\sum_P h_P$. Hence
$d_W(\Gamma)\geq\theta^{-1}\sum_P h_P>(1+\sigma)\sum_Ph_P$, a contradiction.
Together with Lemma~\ref{lem:length}, this proves
\eqref{eq:curve}. The distinctness of the $y_j$ is used only in this final
fibre step; no lower bound on their distances is needed.

\subsection{The passage from curves to coefficient packets}\label{subsec:packets}
We now deduce surjectivity from the curve bound.
Choose an integer $R$ with all $R/w_a$ and $R/v_a$ positive integers.
Embed affine space by all monomials of $W$-degree at most $R$, and let $X$
be the projective closure with $L=\mathcal O_X(1)$.
For the general projective construction from a finite set of monomials,
see \cite[Chapter VI, Section 1]{mondal}.
The constant monomial and all coordinate functions occur in this map,
so it is an embedding of affine space into the chart where the constant
coordinate is nonzero. On the smooth projective model of a curve meeting
that chart, the pole divisor of the coordinate sections has coefficient
\[
 R\max\left(0,-\frac{\ord_PY}{w_0},
                 -\frac{\ord_PX_1}{w_1},\ldots,
                 -\frac{\ord_PX_m}{w_m}\right)
\]
at each point $P$. All monomials satisfy this upper bound. It is attained
by $1$ if the maximum is zero, and otherwise by a pure power of a coordinate
attaining the maximum. Hence the sections have no common zero after
twisting by this divisor, and its degree is the degree of the pullback of
$L$. In particular, $\deg L=R d_W(\Gamma)$ on the curve.

Replace each logarithm by a fixed Taylor polynomial whose first omitted
term has weight strictly greater than $v_i$. With the resulting algebraic
local coordinates $t,u'_1,\ldots,u'_m$, the contact $h_P$ is unchanged.
Indeed, write $u'_i=u_i+\tau_i(t)$ with the first possible tail degree
$T_i$ satisfying $T_iv_0>v_i$. If $t$ is nonzero on the branch, then
\[
 \frac{\ord_P\tau_i(t)}{v_i}
 \geq\frac{T_i\ord_Pt}{v_i}
 >\frac{\ord_Pt}{v_0}\geq h_P.
\]
Thus the tail changes no coordinate attaining the minimum; the same
argument for the inverse substitution proves equality of the minima.
If $t$ vanishes identically, every tail vanishes identically.
Both formal substitutions preserve the $V$-weight filtration, so they
also induce inverse maps on each finite coefficient quotient.
Put
$I_j=(t^{R/v_0},(u'_1)^{R/v_1},\ldots,(u'_m)^{R/v_m})$.
These define a coherent ideal $I$ whose quotient is supported at the centres, and
$\ord_PI=Rh_P$. Blow up this ideal: $p:X'\to X$, with
$I\mathcal O_{X'}=\mathcal O_{X'}(-E)$, and write $A=p^*L$ additively.
Here $E$ is an effective Cartier divisor and
$\mathcal O_{X'}(-E)=\mathcal O_{X'}(1)$ for the blow-up Proj construction;
see \cite[Tag \href{https://stacks.math.columbia.edu/tag/02OS}{02OS}]{stacks}.
Equation~\eqref{eq:curve} implies $A-(1+\sigma)E$ is nef on noncontracted
curves. Contracted curves have positive degree for $-E$, which is relatively
ample by \cite[Tag \href{https://stacks.math.columbia.edu/tag/02NS}{02NS}]{stacks};
curves outside the affine chart avoid $E$.
More explicitly, the normalization of a noncontracted curve upstairs maps
birationally to the normalization of its image, since the blow-up is an
isomorphism away from the finite centre. Pulling back $I$ to this normalized
curve gives an ideal generated locally by powers of a uniformizer, with
order equal to the minimum order of the original generators. Thus
$A\cdot\widetilde\Gamma=R d_W(\Gamma)$ and
$E\cdot\widetilde\Gamma=R\sum_Ph_P$. This checks that the contact quantity
being bounded is precisely the exceptional-divisor intersection, including
all branches of a singular curve.
For some integer $a>1$, $aA-E$ is ample, and
\[
 A-E=\frac{a-1}{a(1+\sigma)-1}(A-(1+\sigma)E)
       +\frac{\sigma}{a(1+\sigma)-1}(aA-E)
\]
is ample by the nef-plus-ample theorem for projective schemes.
The required theorem explicitly allows projective schemes, not just smooth
varieties; see Lazarsfeld~\cite{laz}, Corollary 1.4.10.
For the preceding existence of $aA-E$, choose $b$ with $I\otimes L^b$
globally generated. The blow-up embeds as a closed subscheme of
$X\times\mathbb P^r$ with second hyperplane bundle $bA-E$.
The restriction of the product's ample bundle is $(b+1)A-E$.

The ordinary Rees algebra gives, for sufficiently large $n$,
\[
 p_*\mathcal O_{X'}(-nE)=I^n,\qquad
 R^qp_*\mathcal O_{X'}(-nE)=0\quad(q>0).
\]
See the degree-one Proj results in
\cite[Tags \href{https://stacks.math.columbia.edu/tag/0AG6}{0AG6} and
\href{https://stacks.math.columbia.edu/tag/0AG7}{0AG7}]{stacks}.
The hypotheses hold on every affine open: the base is Noetherian and
$\bigoplus_{n\geq0}I^n$ is generated by finitely many degree-one elements.
The graded module used in 0AG7 is this Rees algebra itself.
There are finitely many affine opens in a cover, so a common eventual
threshold exists. Only eventual equality with the ordinary powers $I^n$
is used; no equality with their integral closures or small-degree equality
is being assumed.
Projection formula, Leray, and Serre vanishing for $A-E$ give
$H^1(X,I^n\otimes L^n)=0$. The precise statements used are
\cite[Tags \href{https://stacks.math.columbia.edu/tag/01E8}{01E8}
(projection formula), \href{https://stacks.math.columbia.edu/tag/01F4}{01F4}
(Leray), and \href{https://stacks.math.columbia.edu/tag/0B5U}{0B5U}
(Serre vanishing)]{stacks}.
For clarity, the Serre theorem is applied to the fixed sheaf
$\mathcal O_{X'}$ and the fixed ample line bundle
$\mathcal L=p^*L\otimes\mathcal O_{X'}(-E)$.
It is not applied directly to the varying sheaves $I^n$ on $X$.
For all sufficiently large $n$, projection formula gives
$p_*\mathcal L^n=I^n\otimes L^n$ and
$R^qp_*\mathcal L^n=0$ for $q>0$; Leray identifies its $H^1$ with that of
$\mathcal L^n$ on $X'$. This checks the fixed-sheaf hypothesis explicitly.
The exact sequence
\[
 0\longrightarrow I^n\otimes L^n\longrightarrow L^n
 \longrightarrow (\mathcal O_X/I^n)\otimes L^n\longrightarrow0
\]
therefore makes global sections of $L^n$ realize all classes modulo $I^n$.
The constant embedding coordinate trivializes $L$ on the affine chart.
The support of $\mathcal O_X/I^n$ is finite, so its global sections are
the direct sum of its local finite-length quotients; completion does not
change these quotients. Each generator of $I$ has $V$-weight $R$, hence
\[
 I_j^n\widehat{\mathcal O}_{X,a_j}
 \subseteq\{\text{formal series of $V$-weight at least $nR$}\}.
\]
This inclusion is in the required direction: the quotient by $I_j^n$
surjects onto the desired smaller coefficient quotient. Equality of these
ideals is neither asserted nor needed.
Finally, for the closed embedding $X\subset\mathbb P^N$, Serre vanishing
for its fixed ideal sheaf $\mathcal I_X$ gives
$H^1(\mathbb P^N,\mathcal I_X(n))=0$ for sufficiently large $n$.
Thus sections of $L^n$ come from homogeneous degree-$n$ polynomials in
the embedding coordinates, hence from affine polynomials of $W$-degree
at most $nR$. This last step does not assume projective normality.
Taking $H=nR$ and undoing the filtered coordinate change gives \eqref{eq:MI}.

This completes the proof of Theorem~\ref{thm:MI}. In particular, its
quantifiers are
\[
 \exists R=R(W,V)\quad\forall\text{ permitted centres }\mathbf a
 \quad\exists n_0(\mathbf a)\quad\forall n\geq n_0(\mathbf a).
\]
The centres will be fixed before this last degree parameter is chosen.

\clearpage
\part{Proof of the main theorem}\label{part:main-proof}
\section{Parameters for a fixed rational argument}\label{sec:parameters}
We begin the proof of Theorem~\ref{thm:main}. It suffices to prove its
inequality for every rational $\nu>2$: for a real exponent $\nu_0>2$,
choose rational $2<\nu<\nu_0$ and use $q^{-\nu}>q^{-\nu_0}$ for $q>1$.
Fix such a rational $\nu$ and write
\[
 r=a/b\ne1,\qquad a,b\in\Z_{>0},\qquad \gcd(a,b)=1,
 \qquad \xi=\log r.
\]
Suppose, towards a contradiction, that there are arbitrarily large
positive integers $q$ with $|\xi-p/q|\leq q^{-\nu}$ for some $p\in\Z$.
All parameters will now be chosen in their order of dependence.

\subsection{The gap and the dimension parameters}
To set the transverse jet weights $w_i/\theta$, choose
\[
 \theta=\frac{5\nu^2-4\nu+4}{(3\nu-2)^2}.
\]
To distinguish small and large transverse indices at the cutoff $AH$, set
\[
 A=\frac{2\nu}{3\nu-2}.
\]
For the fraction of rows with small transverse indices in the two-case
determinant estimate, use the threshold
\[
 \eta=\frac{(\nu-2)^2}{4\nu(3\nu-2)}.
\]
Direct algebra gives
\begin{align*}
 1-\theta&=\frac{4\nu(\nu-2)}{(3\nu-2)^2}>0,&
 A-\theta&=\frac{(\nu-2)(\nu+2)}{(3\nu-2)^2}>0,\\
 \theta-A^2&=\frac{(\nu-2)^2}{(3\nu-2)^2}>0.
\end{align*}
The margin between the principal arithmetic and analytic exponents in
the large-index case is
\[
 g:=\nu(A(1-\eta)-\theta)-(1-\theta)
 =\frac{\nu(\nu-2)^2}{2(3\nu-2)^2}>0.
\]
Also $A<1$ and $0<\eta<1$; for the latter,
$4\nu(3\nu-2)-(\nu-2)^2=11\nu^2-4\nu-4>0$.
To control the growth of the number of centres with the dimension, define
\[
 C=\frac12\left(\frac A\theta+\frac1A\right).
\]
To balance the polynomial weight against that growth, define
\[
 B=\frac12\left(A+\min\{1/C,C\theta\}\right).
\]
Since $A^2<\theta$, they satisfy
\begin{equation}\label{eq:parameter-intervals}
 \frac A\theta<C<\frac1A,\qquad
 A<B<\min\{1/C,C\theta\}.
\end{equation}
In particular, $C>1$, $CB<1$, $C\theta/B>1$, and $B/A>1$.
For the decay forced by repeated Taylor indices, set $c=\log2/4$.
For the least-common-multiple bound used to clear denominators, set
$\Lambda=4\log2$. To bound $|\log r|$ in the analytic estimates, use
$\ell_r=1+\max(a,b)$. To keep the total error below the available margins,
set
\[
 \varepsilon_0=\min(g/2,1/2).
\]
To control the logarithmic truncation orders and their tail error, choose
rational $F>2/\theta$ with $\nu/F<\varepsilon_0/3$.

Let the number $m$ of additive variables and chosen approximations range
over increasing positive integers. Take the number of centres to be
$K=\lfloor C^m\rfloor$, and give the polynomial variable $Y$ the weight
$w_0=B^{-m}$. To satisfy the interpolation volume condition, give the
local variable $t$ the weight
\[
 v_0=2K\theta^m w_0.
\]
The decay from repeated transverse indices will be measured by
\[
 L=\frac{\eta^2K\theta^m}{(m+1)v_0A^m}.
\]
Then $K(w_0/v_0)\theta^m=1/2$, so the volume condition
of Theorem~\ref{thm:MI} holds. Moreover,
\[
 \frac K{w_0}\leq(CB)^m\longrightarrow0,\qquad
 v_0\sim2(C\theta/B)^m\longrightarrow\infty,\qquad
 \frac m{v_0}\longrightarrow0,
\]
and
\[
 L=\frac{\eta^2(B/A)^m}{2(m+1)}\longrightarrow\infty.
\]
The argument $r$ has already been fixed. Hence we can fix $m$ sufficiently
large that
\begin{equation}\label{eq:dimension-choice}
 cL>2,\qquad
 \frac{\Lambda Fm+2\log2}{v_0}
 +\frac{100K\ell_r+(K-1)\log b}{w_0}<\frac{\varepsilon_0}{3}.
\end{equation}

All these parameters are fixed before the approximations are chosen and
the degree tends to infinity.

\subsection{The approximations, weights, and centres}
Under the contradiction assumption, choose $m$ approximations successively
and set
\[
 r_i=p_i/q_i,\qquad w_i=\lceil\log q_i\rceil\geq1,
 \qquad w_* =\min_iw_i.
\]
Choose each denominator large enough to meet the successive weight
thresholds of Theorem~\ref{thm:MI}, and require in addition
\begin{equation}\label{eq:weight-choice}
 \Lambda\sum_{i=1}^m\frac1{w_i}
 +\frac{\theta+\log4+\log(2K)+\nu+\log(200K\ell_r)}{w_*}
 <\frac{\varepsilon_0}{3}.
\end{equation}
This is possible: once all preceding weights are fixed, the next
threshold is finite, and arbitrarily large good denominators remain
available. A common additional lower bound on all $w_i$ ensures
\eqref{eq:weight-choice}.

Use the centres
\[
 y_j=r^j,\qquad c_{ji}=jr_i,\qquad 0\leq j<K.
\]
They have nonzero pairwise distinct $Y$-coordinates because $r>0$ and
$r\ne1$, including when $0<r<1$. Thus Theorem~\ref{thm:MI} applies.
All approximations, weights, and centres are now fixed. Only the degree
$H=nR$ will subsequently tend to infinity.

\section{A rational determinant and its lower bound}\label{sec:arithmetic}
Set
\[
 T_i=\left\lceil\frac{Fw_i}{v_0}\right\rceil,\qquad
 G_i(t)=\sum_{1\leq k<T_i}\frac{(-1)^{k+1}}k t^k,\qquad
 L_i=\lcm(1,\ldots,T_i-1),
\]
where an empty least common multiple is $1$. Since
$v_0T_i>w_i/\theta$, replacing $\log(1+t)$ by $G_i(t)$ is an
invertible filtered substitution on each coefficient quotient.

The interpolation matrix has columns $P=Y^hX^\alpha$ with
$w_0h+\wa\alpha\leq H$, and rows $\rho=(j,s,\beta)$ with
$v_0s+\wa\beta/\theta<H$. Its entries are
\begin{equation}\label{eq:entry}
 [t^su^\beta]\left(r^{jh}(1+t)^h
       \prod_i(jr_i+G_i(t)+u_i)^{\alpha_i}\right).
\end{equation}
For every sufficiently large admissible $H$, Theorem~\ref{thm:MI}
supplies a nonzero square minor $\Delta_H$ containing every row.
All its entries are rational. Write $M=M_H$ for its size and put
\[
 \bar b=\frac1{MH}\sum_\rho\wa\beta,\qquad
 D_H=\prod_iL_i^{\lfloor H/w_i\rfloor}.
\]
Weighted simplex counting gives
\begin{equation}\label{eq:count}
 M\sim\frac{K\theta^mH^{m+1}}{(m+1)!v_0\prod_iw_i},
 \qquad 0\leq\bar b\leq\theta.
\end{equation}

Scale the column indexed by $(h,\alpha)$ and the row indexed by
$(j,s,\beta)$ by, respectively,
\[
 C_{h,\alpha}=b^{(K-1)h}\prod_iq_i^{\alpha_i},\qquad
 S_{j,s,\beta}=D_H\prod_iq_i^{-\beta_i}.
\]
The scaled entry is zero when $\beta\not\leq\alpha$. Otherwise it equals
\begin{equation}\label{eq:integer-entry}
 D_Ha^{jh}b^{(K-1-j)h}\binom\alpha\beta
 [t^s](1+t)^h\prod_i(jp_i+q_iG_i(t))^{\alpha_i-\beta_i}\in\Z.
\end{equation}
Indeed, $K-1-j\geq0$, each $G_i$ has denominator dividing $L_i$, and
$\alpha_i-\beta_i\leq\lfloor H/w_i\rfloor$.
The scaled determinant is therefore a nonzero integer, and hence
\begin{align*}
 \log|\Delta_H|\geq{}&-M\log D_H
 -(K-1)\log b\sum_{\text{columns}}h\\
 &-\sum_{\text{columns}}\sum_i\alpha_i\log q_i
 +\sum_{\text{rows}}\sum_i\beta_i\log q_i.
\end{align*}

We use the elementary bound
$\log\lcm(1,\ldots,n)\leq\Lambda n$, where $\Lambda=4\log2$.
For its proof, every prime power in $(n,2n]$ contributes its prime's
logarithm to $\log\binom{2n}{n}\leq2n\log2$. Summing over dyadic
intervals up to the first power of two at least as large as the endpoint
gives the stated bound. Consequently
\[
 \frac{\log D_H}{H}
 \leq\frac{\Lambda Fm}{v_0}+\Lambda\sum_i\frac1{w_i}.
\]
Since $0\leq w_i-\log q_i<1$, each column satisfies
$\sum_i\alpha_i\log q_i\leq H$ and $h\leq H/w_0$, while
\[
 \sum_{\text{rows}}\sum_i\beta_i\log q_i
 \geq MH\bar b-\frac{MH\theta}{w_*}.
\]
We obtain the lower bound
\begin{equation}\label{eq:arith}
 \frac{\log|\Delta_H|}{MH}\geq-(1-\bar b)-E_{\rm ar},
\end{equation}
where the entire arithmetic cost for the fixed argument $r=a/b$ is
\begin{equation}\label{eq:arithmetic-error}
 E_{\rm ar}=
 \frac{\Lambda Fm}{v_0}+\Lambda\sum_i\frac1{w_i}
 +\frac\theta{w_*}+\frac{(K-1)\log b}{w_0}.
\end{equation}

\section{Translation and analytic determinant bounds}\label{sec:analytic}
\subsection{Exact translation of every row}
For each selected column $P=Y^hX^\alpha$ and multi-index $\lambda$, define
the entire function
\[
 f_{\lambda,P}(z)=[u^\lambda]P(e^z,z+u_1,\ldots,z+u_m)
 =\begin{cases}
 \binom\alpha\lambda e^{hz}z^{|\alpha|-|\lambda|},&\lambda\leq\alpha,\\
 0,&\text{otherwise}.
 \end{cases}
\]
Set
\[
 z_j(t)=j\xi+\log(1+t),\qquad
 \epsilon_{ji}=j(r_i-\xi),\qquad
 \tau_i(t)=G_i(t)-\log(1+t).
\]
The exact identities
\[
 e^{z_j(t)}=r^j(1+t),\qquad
 jr_i+G_i(t)+u_i=z_j(t)+u_i+\epsilon_{ji}+\tau_i(t)
\]
give the following finite expansion as an equality of row vectors:
\begin{align}\label{eq:translation}
 \operatorname{row}_{j,s,\beta}
 &=\sum_{\substack{\lambda\geq\beta\\\wa\lambda\leq H}}
    \binom\lambda\beta
   \sum_{0\leq d\leq\lambda-\beta}\binom{\lambda-\beta}d
      \epsilon_j^{\lambda-\beta-d}\nonumber\\
 &\quad\cdot\sum_{k=0}^s
  \left([t^k]\prod_i\tau_i(t)^{d_i}\right)
  \bigl([t^{s-k}]f_{\lambda,P}(z_j(t))\bigr)_P.
\end{align}
All scalar factors are independent of the column. A nonzero tail
coefficient requires $\sum_i d_iT_i\leq k\leq s$, so
$\wa d<H/F$. On $|t|=1/2$ we have $|\tau_i(t)|\leq1$, and
\[
 |\epsilon_{ji}|\leq Kq_i^{-\nu}\leq Ke^\nu e^{-\nu w_i}.
\]
The two binomial factors in \eqref{eq:translation} have product at most
$4^{|\lambda|}$, and $|\lambda|\leq H/w_*$. Cauchy's coefficient
estimate bounds the tail coefficient by $2^k\leq2^{H/v_0}$.
Combining these estimates, every scalar multiplying a test row is at most
\begin{equation}\label{eq:scalar}
 \exp\{-\nu\wa{(\lambda-\beta)}+HE_{\rm tr}\},\qquad
 E_{\rm tr}=\frac\nu F+\frac{\log2}{v_0}
              +\frac{\log4+\log(2K)+\nu}{w_*}.
\end{equation}
The harmless factor $2$ in $2K$ allows a uniform upper bound.
Because $w_i\geq1$, the number of choices $(\lambda,d,k)$ for each row
is at most
\[
 Q_H=(\lfloor H\rfloor+1)^{2m}(\lfloor H/v_0\rfloor+1).
\]

\subsection{Repeated transverse indices and Taylor coefficients}
Use the radius
\[
 R_r=100K\ell_r,\qquad \ell_r=1+\max(a,b).
\]
The inequality $\log x\leq x-1$, applied to $r$ or $r^{-1}$, gives
$|\xi|<\ell_r$. Hence, on $|t|\leq1/2$,
\[
 |z_j(t)|\leq(K-1)|\xi|+\log2<K\ell_r<R_r/2.
\]
For every allowed column and multi-index, on $|z|\leq R_r$,
\[
 |f_{\lambda,P}(z)|\leq
 \exp\left\{H\left(\frac{R_r}{w_0}
                 +\frac{\log(2R_r)}{w_*}\right)\right\}.
\]
Let $T$ be any matrix obtained by choosing one unscaled test row from
\eqref{eq:translation} for each original row, and let $n_\lambda$ count
the chosen rows with transverse index $\lambda$. We claim
\begin{equation}\label{eq:collision}
 |\det T|\leq
 \exp\left\{-c\sum_\lambda n_\lambda^2
                  +MH(E_{\rm hol}+\rho_H)\right\},
\end{equation}
where
\[
 E_{\rm hol}=\frac{100K\ell_r}{w_0}+\frac{\log2}{v_0}
                         +\frac{\log(200K\ell_r)}{w_*},
 \qquad \rho_H=\frac{\log M+3}{H}.
\]
To prove this, expand each $f_{\lambda,P}$ in powers of $(z/R_r)^d$.
Its Taylor coefficients have the same bound as the displayed supremum.
The test $[t^\ell]$ applied to $(z_j(t)/R_r)^d$ has modulus at most
$2^\ell2^{-d}$. If two chosen pairs $(\lambda,d)$ coincide, the
corresponding coefficient rows are identical and their determinant is zero.
Every nonzero term therefore has
\[
 \sum_\rho d_\rho\geq\sum_\lambda\binom{n_\lambda}{2}.
\]
Use half of $2^{-\sum d_\rho}$ for this inequality and sum the remaining
half geometrically. If $\mathcal D$ denotes the possible assignments in
nonzero terms, this gives
\[
 \sum_{(d_\rho)\in\mathcal D}2^{-\sum_\rho d_\rho}
 \leq\exp\left\{-c\sum_\lambda n_\lambda^2+cM\right\}
       (2+\sqrt2)^M.
\]
The coefficient determinant is bounded by $M!$ times the product of its
entry bounds. The factors $2^{\ell_\rho}$ have product at most
$\exp(MH\log2/v_0)$. Finally,
$\log M!\leq M\log M$ and $c+\log(2+\sqrt2)<3$.
These estimates prove \eqref{eq:collision}, uniformly over the selected
minor and all row choices. They also justify absolute convergence of the
infinite Taylor expansions. The constants $\log2$ here arise from the
circle $|t|=1/2$ and the geometric series estimates.

\section{The determinant contradiction}\label{sec:contradiction}
The total normalized error is
\begin{align}\label{eq:total-error}
 E&=E_{\rm ar}+E_{\rm tr}+E_{\rm hol}\nonumber\\
 &=\frac\nu F+\frac{\Lambda Fm+2\log2}{v_0}
       +\frac{100K\ell_r+(K-1)\log b}{w_0}\nonumber\\
 &\quad+\Lambda\sum_i\frac1{w_i}
       +\frac{\theta+\log4+\log(2K)+\nu+\log(200K\ell_r)}{w_*}.
\end{align}
The choices of $F$, \eqref{eq:dimension-choice}, and
\eqref{eq:weight-choice} give
\begin{equation}\label{eq:conditions}
 E<\varepsilon_0\leq\min(g/2,1/2),\qquad cL>2>1+E.
\end{equation}
We now keep every parameter fixed and let $H$ increase through the
admissible multiples of Theorem~\ref{thm:MI}.

For bounds valid at every degree, put
\[
 S=\sum_iw_i,\qquad
 L_H=L\left(1+\frac{S}{AH}\right)^{-m},\qquad
 \delta_H=\frac{\log M+\log Q_H+3}{H}.
\]
Unit-cube comparisons with weighted simplices give
\begin{align}\label{eq:lattice-bounds}
 M&\geq\frac{K\theta^mH^{m+1}}{(m+1)!v_0\prod_iw_i},\nonumber\\
 D_A(H):=\#\{\lambda\in\N^m:\wa\lambda\leq AH\}
 &\leq\frac{(AH+S)^m}{m!\prod_iw_i}.
\end{align}
For the first bound, flooring every point of the strict jet simplex gives
an admissible lattice point. For the second, the half-open unit cube based
at an admissible $\lambda$ lies in $\sum_iw_ix_i<AH+S$.

Expand $\Delta_H$ by multilinearity using \eqref{eq:translation}.
For each choice of test rows, one of two alternatives holds.
If at least $\eta M$ chosen indices satisfy $\wa\lambda\leq AH$, then
Cauchy--Schwarz and \eqref{eq:lattice-bounds} imply
\[
 \frac{\sum_\lambda n_\lambda^2}{MH}
 \geq\frac{\eta^2M}{HD_A(H)}\geq L_H.
\]
All scalar main exponents in \eqref{eq:scalar} are nonpositive because
$\lambda\geq\beta$. Equations~\eqref{eq:scalar} and \eqref{eq:collision}
therefore bound this term by a normalized exponent
$-cL_H+E_{\rm tr}+E_{\rm hol}+\rho_H$.

Otherwise more than $(1-\eta)M$ chosen indices have weight greater than
$AH$, so
\[
 \sum_\rho\wa{\lambda_\rho}>MHA(1-\eta).
\]
Discarding the nonpositive quadratic term in \eqref{eq:collision}, the
normalized exponent of this term is at most
$-\nu(A(1-\eta)-\bar b)+E_{\rm tr}+E_{\rm hol}+\rho_H$.
There are at most $Q_H^M$ terms in the expansion. Thus
\begin{equation}\label{eq:upper}
 \frac{\log|\Delta_H|}{MH}
 \leq E_{\rm tr}+E_{\rm hol}+\delta_H
       +\max\{-cL_H,-\nu(A(1-\eta)-\bar b)\}.
\end{equation}
Compare this with the lower bound \eqref{eq:arith}. Its excess over the
low-index expression is at least $cL_H-1-E-\delta_H$. For the high-index
expression, use $\bar b\leq\theta$ and $\nu>1$ to obtain
\[
 \nu(A(1-\eta)-\bar b)-(1-\bar b)
 \geq\nu(A(1-\eta)-\theta)-(1-\theta)=g.
\]
The lower bound exceeds that expression by at least $g-E-\delta_H$.
Hence the bounds are incompatible whenever
\begin{equation}\label{eq:finite-conditions}
 E+\delta_H<g,\qquad cL_H>1+E+\delta_H.
\end{equation}

For fixed parameters, $L_H\to L$ and $\delta_H\to0$, since $M$ and $Q_H$
grow at most polynomially in $H$. In view of \eqref{eq:conditions}, both
inequalities in \eqref{eq:finite-conditions} hold for every sufficiently
large admissible $H$. The nonzero minor supplied by Theorem~\ref{thm:MI}
therefore leads to a contradiction. The interpolation threshold may depend
on the centres; all of them were fixed before taking this limit.

It follows that for the fixed pair $(r,\nu)$ there are only finitely many
denominators admitting $|\log r-p/q|\leq q^{-\nu}$. This proves the
asserted threshold $Q(r,\nu)$, and the reduction at the start extends it
to every real $\nu>2$. If $\xi=p_0/q_0$ were rational, its exact
representations $(np_0)/(nq_0)$ would contradict this threshold for
arbitrarily large denominators. Thus $\xi$ is irrational.
Dirichlet's theorem gives $\mu(\xi)\geq2$, and the proved inequalities
give $\mu(\xi)\leq2$. This completes the proof of Theorem~\ref{thm:main}.
\qed

\clearpage
\part{Consequences}\label{part:consequences}
\section{Rational linear combinations and affine changes}\label{sec:consequences}
\begin{lemma}[Invariance under rational affine changes]\label{lem:affine-invariance}
If $x$ is irrational, $u\in\mathbb Q\setminus\{0\}$ and
$v\in\mathbb Q$, then $\mu(ux+v)=\mu(x)$.
\end{lemma}
\begin{proof}
Write $u=A_1/B_1$ and $v=A_2/B_2$ with $B_1,B_2>0$.
A rational approximation $p/q$ to $x$ gives the rational number
\[
 \frac{A_1B_2p+A_2B_1q}{B_1B_2q}
\]
approximating $ux+v$ with error $|u|\,|x-p/q|$. Its reduced denominator
$Q$ is at most $B_1B_2q$. For any exponent $s>0$, an error at most
$q^{-s}$ consequently gives an error at most
$|u|(B_1B_2)^sQ^{-s}$. Along infinitely many approximations tending to
the irrational number $ux+v$, the denominators $Q$ are unbounded.
For every $s'<s$, the fixed factor is absorbed by $Q^{s-s'}$ for all
sufficiently large $Q$. Taking suprema of approximation exponents gives
$\mu(ux+v)\geq\mu(x)$. Applying the same argument to the inverse affine
map gives the opposite inequality, also when either exponent is infinite.
\end{proof}

\begin{corollary}[Nonzero rational logarithmic combinations]\label{cor:log-combinations}
Let $r_1,\ldots,r_s$ be positive rational numbers and
$c_1,\ldots,c_s$ rational numbers. If
\[
 \xi=\sum_{k=1}^s c_k\log r_k\ne0,
\]
then $\mu(\xi)=2$. More generally, $\mu(u\xi+v)=2$ for every
$u\in\mathbb Q\setminus\{0\}$ and $v\in\mathbb Q$.
\end{corollary}
\begin{proof}
Choose a positive integer $D$ with $Dc_k\in\mathbb Z$ for every $k$.
The number
\[
 r_* =\prod_{k=1}^s r_k^{Dc_k}
\]
is positive and rational, including when some exponents are negative.
The real logarithm identity gives $D\xi=\log r_*$ exactly. Since
$\xi\ne0$, $r_*\ne1$. Theorem~\ref{thm:main} and
Lemma~\ref{lem:affine-invariance} now give all the conclusions.
The nonzero condition cannot be replaced merely by requiring that some
$c_k\ne0$, because different logarithmic terms can cancel.
\end{proof}

\begin{corollary}[Positive rational powers]\label{cor:rational-powers}
If $\alpha>0$, $\alpha\ne1$, and
$\alpha^D\in\mathbb Q$ for some positive integer $D$, then
$\mu(\log\alpha)=2$.
\end{corollary}
\begin{proof}
Use $D\log\alpha=\log(\alpha^D)$ and the preceding theorem and lemma.
\end{proof}

\section{Scope and effectivity}\label{sec:scope}
The distinct values $r^j$ are the geometric feature that permits the
direct logarithmic argument. Theorem~\ref{thm:MI} allows arbitrary
additive coordinates because a curve with constant multiplicative
coordinate meets at most one centre. Appendix~\ref{app:fibres} treats
repeated multiplicative coordinates under an additional fibre-volume
condition and separation of the additive coordinates within each fibre.
Appendix~\ref{app:volume-control} shows that the volume inequality alone
cannot replace the separation hypotheses.

The integer determinant argument uses that $r$ is rational. For a general
algebraic argument, clearing denominators would give an algebraic integer
whose size in one embedding need not be bounded below by $1$. A norm
argument would therefore require estimates for the conjugate determinants
as well. No such estimates are used here. Corollary~\ref{cor:rational-powers}
covers the indicated algebraic arguments through an exact logarithm
identity, without this additional arithmetic step.

The proof is qualitative in the denominator threshold $Q(r,\nu)$.
After the approximations and their weights have been chosen, the final
limit requires only the existence of an interpolation threshold for that
fixed centre configuration. Appendix~\ref{app:degree} gives an explicit
degree bound for the analytic contradiction once all numerical parameters
are fixed. It does not estimate the centre-dependent Serre-vanishing
threshold in Theorem~\ref{thm:MI}, and therefore does not make the
denominator threshold in Theorem~\ref{thm:main} effective.

\clearpage
\appendix
\part*{Appendices}
\phantomsection
\addcontentsline{toc}{part}{Appendices}
\section{An explicit bound for the final degree}\label{app:degree}
This bound quantifies the analytic and counting errors after all parameters
have been fixed. It does not bound the geometric interpolation threshold.
\begin{proposition}\label{prop:finite-degree}
In the setting of Section~\ref{sec:contradiction}, let
\[
 \gamma=\min\{g-E,cL-1-E\}>0,\qquad
 a_0=\max\{1,v_0^{-1}\},\qquad C_0=\log K+2\log a_0+3.
\]
Every $H$ satisfying
\begin{equation}\label{eq:analytic-threshold}
 H>\max\left\{1,\frac{2cmLS}{A\gamma},
               \frac{16(3m+2)^2}{\gamma^2},\frac{4C_0}{\gamma}\right\}
\end{equation}
satisfies \eqref{eq:finite-conditions}. These bounds are independent of
the selected nonzero minor and its row expansion.
\end{proposition}
\begin{proof}
The box bounds and $w_i\geq1$ give
\begin{align*}
 M&\leq K(H/v_0+1)\prod_i(\theta H/w_i+1)
      \leq K a_0(H+1)^{m+1},\\
 Q_H&\leq a_0(H+1)^{2m+1}.
\end{align*}
Consequently, for $H\geq1$,
\[
 \delta_H\leq
 \frac{(3m+2)\log(H+1)+C_0}{H}
 \leq\frac{3m+2}{\sqrt H}+\frac{C_0}{H}.
\]
Here $\log(1+H)\leq\sqrt H$: the difference has nonnegative derivative
because $H+1\geq2\sqrt H$, and is positive at $H=1$.
The last two terms in the maximum \eqref{eq:analytic-threshold} imply
$\delta_H<\gamma/2$.
The elementary inequality $(1+x)^{-m}\geq1-mx$ for $x\geq0$ gives
\[
 c(L-L_H)\leq\frac{cmLS}{AH}<\frac\gamma2.
\]
It follows that
$cL_H>cL-\gamma/2\geq1+E+\gamma/2>1+E+\delta_H$ and
$E+\delta_H<E+\gamma/2<g$, as required.
\end{proof}

\section{Interpolation with bounded fibre multiplicity}\label{app:fibres}
The following extension records the more general geometric statement.
The direct proof of Theorem~\ref{thm:main} uses only the distinct-$Y$
case already established in Theorem~\ref{thm:MI}.

\begin{theorem}[Separated-weight interpolation with bounded fibre multiplicity]
\label{thm:fibre-interpolation}
Fix integers $m,K\geq1$ and $1\leq\kappa\leq K$, and positive rational
$w_0,v_0,\theta$ satisfying
\begin{equation}\label{eq:fibre-volumes}
 0<\theta<1,\qquad
 K(w_0/v_0)\theta^m<1,\qquad \kappa\theta^m<1.
\end{equation}
There are successive lower thresholds for positive rational
$w_1,\ldots,w_m$, depending only on these fixed data and earlier weights,
with the following property. Set
$W=(w_0,\ldots,w_m)$ and $V=(v_0,w_1/\theta,\ldots,w_m/\theta)$.
For any centres
\[
 a_j=(y_j,c_{j1},\ldots,c_{jm})\in\C^*\times\C^m,\qquad 0\leq j<K,
\]
assume that at most $\kappa$ indices have any given value of $y_j$ and that
\begin{equation}\label{eq:within-fibre}
 j\ne l,\quad y_j=y_l
 \quad\Longrightarrow\quad
 c_{ji}\ne c_{li}\quad\hbox{for every }1\leq i\leq m.
\end{equation}
Then the map \eqref{eq:MI} is onto for every $H=nR$ with $n\geq n_0$.
Here a positive integer $R$ can be chosen using only $W,V$, before the
centres are specified, whereas $n_0$ may depend on the centres as well as
all previously fixed data.
\end{theorem}

\begin{proof}
Choose rational $\sigma>0$ so that
\begin{equation}\label{eq:fibre-sigma}
 \begin{gathered}
 (1+3\sigma)^{m+1}K(w_0/v_0)\theta^m<1,\\
 (1+3\sigma)^m\kappa\theta^m<1,\qquad (1+\sigma)\theta<1.
 \end{gathered}
\end{equation}
Such a choice exists by the strict inequalities \eqref{eq:fibre-volumes}.
Using this choice of $\sigma$, take
\[
 \ell_0=m+2,\qquad
 C_*=\max_{1\leq k\leq m}2k!(\ell_0/\sigma)^k,
\]
as in Section~\ref{subsec:moving-curves}.
Enforce \eqref{eq:separation} by the
successive weight choices already justified there. This choice of weights
has no dependence on any centre.

We prove the curve inequality \eqref{eq:curve} for the present centres.
Suppose an irreducible curve $\Gamma$ violates it. Dimension counting
produces $F_N$ for large $N$: the ratio of conditions to coefficients
tends to
$(1+3\sigma)^{m+1}K(w_0/v_0)\theta^m<1$.
By Lemma~\ref{lem:moving-chart}, the derivative costs and degree bounds are
identical at all the centres. The zero--pole argument, persistent-component
construction, and normal-basis comparison of
Section~\ref{subsec:moving-curves} therefore apply unchanged up to the
conclusion that $Y$ is constant on $\Gamma$.
Explicitly, the persistent component has
$\prod_{a\in A}w_a\leq C_*\prod_{b\in B}v_b$ for every pair of normal bases.
Together with \eqref{eq:separation}, this either puts all positive coordinate
directions in its tangent space or forces $dX_i=dY/Y$ for some $i$.
The residue argument in that section forces $Y$ to be constant in the
second case as well. No distinctness of the $y_j$ or of the $c_{ji}$ was
used in this part of the proof.

Write this constant as $y$. Since a violating curve meets a centre,
$y\ne0$ and
$J_y=\{j:y_j=y\}$ is nonempty, with cardinality at most $\kappa$.
On the fibre $Y=y$ the contacts are
\[
 h_P=\min_i\frac{\ord_P(X_i-c_{ji})}{v_i}
 \quad(P\text{ lying over }a_j,\ j\in J_y).
\]
All prescribed centres on $\Gamma$ belong to this one fibre. If
$|J_y|=1$, the final zero--pole comparison below applies immediately.
Otherwise use the following argument, which also works when $|J_y|=1$.

In the $m$ variables $X_i$, choose a nonzero polynomial $G_N$ of weighted
degree at most $N$ for the weights $(w_1,\ldots,w_m)$,
with $V'=(v_1,\ldots,v_m)$-order at least
$(1+3\sigma)N$ at all points $(c_{ji})_i$ with $j\in J_y$.
The ratio of conditions to coefficients tends to
\[
 |J_y|(1+3\sigma)^m
 \frac{\prod_i w_i}{\prod_i v_i}
 =|J_y|(1+3\sigma)^m\theta^m<1
\]
by \eqref{eq:fibre-sigma}. Every partial derivative of total cost at most
$\sigma N$ has order at least $(1+2\sigma)Nh_P$ on each relevant branch
and pole degree at most $Nd_W(\Gamma)$. It must vanish on $\Gamma$.
When $m=1$ this is impossible already for $G_N$ itself: a nonzero
polynomial on an affine line cannot vanish on a curve.

For $m\geq2$, use the common-zero sets in this fibre of all
$\partial_X^\gamma G_N$ with cost at most $r\sigma N/\ell_0$,
$0\leq r\leq\ell_0$. Choose nested irreducible components containing the
affine part of $\Gamma$, starting at the last level and proceeding
backwards. Their dimensions are between $1$ and $m-1$; two adjacent
components therefore coincide in $Z'$. The equations at the earlier
level, followed by any derivative of cost less than $\sigma N/\ell_0$,
vanish on $Z'$. Apply Lemma~\ref{lem:length} in this affine $m$-space,
with degree weights $w_i$, frame $\partial_{X_i}$, and costs $v_i$.
For all coordinate normal bases $A,B$ of $Z'$ it gives
\[
 \prod_{a\in A}w_a\leq C_*\prod_{b\in B}v_b.
\]
There can be only one normal-basis index set. Otherwise order two
different such sets so that the largest index in their symmetric
difference belongs to $A$; \eqref{eq:separation} is then a contradiction.
The unique set is nonempty because $Z'$ has positive codimension.
Choose $i$ in it. All other coordinate directions fail to span the
normal quotient, so $T Z'\subseteq\ker dX_i$ at its general smooth points.
Consequently $X_i$ is constant on $Z'$.
Condition \eqref{eq:within-fibre} implies that $Z'$ meets at most one
prescribed centre in this fibre. Hence $\Gamma$ also meets at most one.

Let that centre be $a_j$. Some $X_l$ is nonconstant on $\Gamma$, since
$Y$ is constant and $\Gamma$ is a curve. The rational function
$X_l-c_{jl}$ has pole degree at most $w_l d_W(\Gamma)$, and its zeros
over $a_j$ have total order at least
$(w_l/\theta)\sum_P h_P$. The zero--pole equality now gives
\[
 d_W(\Gamma)\geq\theta^{-1}\sum_Ph_P
 >(1+\sigma)\sum_Ph_P,
\]
contradicting the assumed violation. This proves the curve inequality
for every curve and every permitted centre configuration.

The centres are distinct as full points, by \eqref{eq:within-fibre}.
Choose $R$ with $R/w_a,R/v_a$ positive integers, using only the weights.
The Taylor cutoffs, weighted projective embedding, and local ideals
from Section~\ref{subsec:packets} are therefore available.
For each fixed configuration the curve inequality makes
$A-(1+\sigma)E$ nef, hence $A-E$ ample. Eventual recovery of the ordinary
Rees powers and Serre vanishing give
$H^1(X,I^n\otimes L^n)=0$ for all $n$ beyond a finite threshold.
Sections of $L^n$ come from embedding polynomials for all sufficiently
large $n$ as well. Taking the maximum of these thresholds gives an $n_0$
such that every $H=nR$, $n\geq n_0$, has the stated surjectivity.
All details of the ordinary-power comparison and passage to the
weighted coefficient quotient were proved in Section~\ref{subsec:packets}.
The possibly centre-dependent choices of $I$, the blow-up, and the
vanishing thresholds occur only after the weights and $R$ have been fixed.
\end{proof}

\begin{corollary}[The fixed-fibre and moving-centre cases]\label{cor:MI}
Theorem~\ref{thm:fibre-interpolation} recovers \cite[Theorem 2.1]{source}
by taking $y_j=1$ for all $j$ and $\kappa=K$.
If the $y_j$ are pairwise distinct, take $\kappa=1$.
Then $\kappa\theta^m<1$ is automatic and
\eqref{eq:within-fibre} is vacuous. Thus Theorem~\ref{thm:MI} is recovered with no restrictions
on the $c_{ji}$ beyond being complex numbers.
\end{corollary}
\begin{proof}
These substitutions give precisely the indicated hypotheses and maps.
In the distinct-$Y$ case, a curve with constant $Y$ already meets at most
one centre, so the second auxiliary polynomial $G_N$ can be omitted.
\end{proof}

\section{Why the volume condition alone is insufficient}\label{app:volume-control}
There is a concrete counterexample to dropping weight separation from MI.
Take $m=K=1$, centre $(y,c)=(1,0)$, and
\[
 W=(1,1),\qquad \theta=4/5,\qquad V=(9/10,5/4).
\]
The volume ratio is $(w_0/v_0)\theta=8/9<1$.
Consider the algebraic curve $X=Y-1$. In the jet coordinates it is
\[
 Y=1+t,\qquad X=t,\qquad
 u=t-\log(1+t)=\tfrac12t^2+O(t^3).
\]
For $H=45n$, substitution on this curve gives a well-defined surjection
\[
 \mathcal J_V(45n)\longrightarrow\C[[t]]/(t^{50n}).
\]
Indeed, a discarded monomial satisfies
$(9/10)s+(5/4)\beta\geq45n$, which implies
$s+2\beta\geq50n$; and all $t^s$ with $s<50n$ are retained.
But the restriction of any $P\in\mathcal P_W(45n)$ is $P(1+t,t)$,
a polynomial of degree at most $45n$.
Its image has dimension at most $45n+1<50n$ for every $n\geq1$.
Hence this interpolation map fails to be onto for every such $H$,
despite the strict volume inequality.

This example violates weight separation and therefore does not contradict
Theorem~\ref{thm:MI}. It shows why the volume inequality cannot replace
the geometric proof in Section~\ref{sec:geometry}.

\end{document}